\documentclass[12]{amsart}
\usepackage{amsfonts}
\usepackage{mathrsfs}
\usepackage{amsmath}
\usepackage{amssymb}
\usepackage{graphicx}
\newtheorem{thm}{Theorem}[section]
\newtheorem{cor}[thm]{Corollary}
\newtheorem{lem}[thm]{Lemma}

\newtheorem{prop}[thm]{Proposition}
\newtheorem{exa}[thm]{Example}

\theoremstyle{definition}
\newtheorem{defn}[thm]{Definition}

\numberwithin{equation}{section}

\begin{document}
\vbox{\vskip 1cm}

\title{\small Characterizations of $I$-semiregular and $I$-semiperfect rings }
\maketitle
\begin{center}
Yongduo Wang
\vskip 4mm Department of Applied Mathematics, Lanzhou University of Technology\\
Lanzhou 730050, P. R. China\\
 E-mail: ydwang@lut.cn

\end{center}

\bigskip
\leftskip10truemm \rightskip10truemm \noindent {\bf Abstract}
 Let $M$ be a left module over a ring $R$ and $I$ an ideal of $R$. We call $(P, f)$ a
 (locally)projective $I$-cover of $M$ if $f$ is an epimorphism from $P$ to $M$, $P$ is (locally)projective, $Kerf\subseteq IP$,
 and whenever $P=Kerf+X$, then there is a projective summand $Y$ of
 $P$ in $Kerf$ such that $P=Y\oplus X$. This definition generalizes
 (locally)projective covers. We characterize
 $I$-semiregular and $I$-semiperfect rings which are defined by
 Yousif and Zhou [19] using (locally)projective $I$-covers in section 2 and 3. $I$-semiregular and $I$-semiperfect
 rings are characterized by projectivity classes in section 4. Finally, the notion of $I$-supplemented modules are
 introduced and $I$-semiregular and $I$-semiperfect
 rings are characterized by $I$-supplemented modules. Some well known results are obtained as corollaries.\\


\leftskip0truemm \rightskip0truemm
\bigskip

\baselineskip=20pt

\section{\bf Introduction and Preliminaries}

\noindent It is well known that (locally) projective covers,
projectivity classes and supplemented modules play important roles
in characterizing semiperfect and semiregular rings. Recently, some
authors had worked with various extensions of these rings (see for
examples [2, 9, 11, 19, 20]). As generalizations of semiregular
rings and semiperfect rings, the notions of $I$-semiregular rings
and $I$-semiperfect rings were introduced by Yousif and Zhou [19].
Our purposes of this paper as follows: (1) characterize
$I$-semiregular and $I$-semiperfect rings by defining (locally)
projective $I$-covers in section 2 and 3; (2) characterize
$I$-semiregular rings and $I$-semiperfect rings in term of
projectivity classes of modules in section 4. (3) chracterize
$I$-semiregular rings and $I$-semiperfect rings by defining
$I$-supplemented modules in section 5.

Let $R$ be a ring and $I$ an ideal of $R$, $M$ a module and $S\leq
M$. $S$ is called \emph{small} in $M$ (notation $S\ll M$) if $M\neq
S + T$ for any proper submodule $T$ of $M$.  As a proper
generalization of small submodules, the concept of $\delta$-small
submodules was introduced by Zhou in [20]. Let $N\leq M$. $N$ is
said to be \emph{$\delta$-small} in $M$ if, whenever $N+X=M$ with
$M/X$ singular, we have $X=M$. $\delta(M)=Rej_{M}(\wp)=\cap\{N\leq
M\mid M/N\in\wp\}$ , where $\wp$ be the class of all singular simple
modules. We also recall that a pair $(P, f)$ is called \emph{a
(locally) projective $(\delta-)$cover} of $M$ if $P$ is (locally)
projective, $f$ is an epimorphism from $P$ to $M$ such that $Kerf\ll
P$($Kerf\ll_{\delta} P$). An element $m$ of $M$ is called
\emph{$I$-semiregular} [2] if there exists a decomposition
$M=P\oplus Q$ where $P$ is projective, $P\subseteq Rm$ and $Rm\cap
Q\subseteq IM$. $M$ is called \emph{an $I$-semiregular module} if
every element of $M$ is $I$-semiregular. $R$ is called
\emph{$I$-semiregular} if $_{R}R$ is an $I$-semiregular module. Note
that $I$-semiregular rings are left-right symmetric and $R$ is
($\delta-$) semiregular if and only $R$ is ($\delta(_{R}R)-)$
$J(R)$-semiregular. $M$ is called \emph{an $I$-semiperfect module}
[11] if for every submodule $K$ of $M$, there is a decomposition
$M=A\oplus B$ such that $A$ is projective, $A\subseteq K$ and $K\cap
B\subseteq IM$. $R$ is called \emph{$I$-semiperfect} if $_{R}R$ is
an $I$-semiperfect module. Note that $R$ is ($\delta-$) semiperfect
if and only $R$ is ($\delta(_{R}R)-)$ $J(R)$-semiperfect. For other
standard definitions we refer to [3, 10, 13].

In this note all rings are associative with identity and all modules
are unital left modules unless specified otherwise. Let $R$ be a
ring and $M$ a module. We use $Rad(M)$, $Soc(M)$, $Z(M)$ to indicate
the Jacobson radical, the socle, the singular submodule of $M$
respectively. $J(R)$ is the radical of $R$ and $I$ is an ideal of
$R$.

\section{\bf $I$-semiregular($I$-semiperfect) rings and projective $I$-covers}

In this section, we introduce the notion of PSD submodules of
modules and use this to define projective $I$-covers which are a
generalization of some well-known projective covers.
Characterizations of $I$-semiregular and $I$-semiperfect rings are
given by projective $I$-covers.  We begin this section with the
following definitions.

\begin{defn} Let $I$ be an ideal of $R$ and $N\leq M$. $N$ is PSD in
$M$ if there exists a projective summand $S$ of $M$ such that $S\leq
N$ and $M=S\oplus X$ whenever $N+X=M$ for any submodule $X\leq M$.
$M$ is PSD for $I$ if any submodule of $IM$ is PSD in $M$. $R$ is a
left PSD ring for $I$ if any finitely generated free left $R$-module
is PSD for $I$.

\end{defn}

\begin{lem} Let $N$ be a direct summand of $M$ and $A\leq N$. Then $A$
is PSD in $M$ if and only if $A$ is PSD in $N$.
\end{lem}
\begin{proof} $``\Rightarrow"$ Since $N$ is a direct summand of $M$, $M=N\oplus K$
for some submodule $K\leq M$. Suppose that $N=A+X, X\leq N$, then
$M=A+(X\oplus K)$. Since $A$ is PSD in $M$, there is a projective
direct summand $Y$ of $M$ such that $Y\leq A$ and $M=Y\oplus X\oplus
K$, and hence $N=N\cap M=X\oplus Y$.

$``\Leftarrow"$ Let $M=A+L, L\leq M$. Then $N=N\cap M=A+N\cap L$.
Since $A$ is PSD in $N$, there is a projective summand $K$ of $N$
with $K\leq A$ such that $N=K\oplus (N\cap L)$. It is easy to see
that $K\cap L=0$. Next we only show that $M=K+L$. Let $m\in M$, then
$m=a+l$, $a\in A, l\in L$. Since $a=k+s, k\in K, s\in N\cap L$,
$m=k+s+l$. Note that $s+l\in L$, so $m\in K+L$, and hence $M=K+L$,
as required.

\end{proof}

\begin{cor} Let $M$ be a $R$-module. If $M$ is PSD for an ideal $I$ of
$R$, then any direct summand of $M$ is PSD for $I$.
\end{cor}

\begin{cor} A ring $R$ is a left PSD ring for an ideal $I$ if and
only if any finitely generated projective left $R$-module is PSD for
$I$.
\end{cor}

\begin{prop} Let $M=M_{1}\oplus M_{2}$. If $N_{1}$ is PSD in $M_{1}$
and $N_{2}$ is PSD in $M_{2}$, then $N_{1}\oplus N_{2}$ is PSD in
$M$.
\end{prop}

\begin{proof} Let $M=N_{1}\oplus N_{2} +L, L\leq M$. Since $N_{1}$ is PSD in
$M_{1}$, $N_{1}$ is PSD in $M$ by Lemma 2.2. Thus there is a
projective summand $S_{1}$ of $M$ with $S_{1}\subseteq N_{1}$ such
that $M=S_{1}\oplus (N_{2}+L)$. Similarly, there exists a projective
summand $S_{2}$ of $M$ with $S_{2}\subseteq N_{2}$ such that
$M=S_{1}\oplus S_{2}+L$. The rest is obvious.
\end{proof}

\begin{defn} A pair $(P, f)$ is called a projective $I$-cover of $M$
if $P$ is projective, $f$ is an epimorphism from $P$ to $M$ such
that $Kerf\leq IP$, and $Kerf$ is PSD in $P$.

\end{defn}

It is easy to see that a module $M$ has a projective
$\delta(_{R}R)$-cover (projective $J(R)$-cover, respectively) if and
only if $M$ has a projective $\delta$-cover (projective cover,
respectively) by [1, Proposition 3.6].

\begin{prop} If each $f_{i}: P_{i}\rightarrow M_{i}, (i=1, 2,
\cdot\cdot\cdot, n)$ is a projective $I$-cover, then $\oplus
_{i=1}^{n}f_{i}: \oplus _{i=1}^{n}P_{i}\rightarrow \oplus
_{i=1}^{n}M_{i}$ is a projective $I$-cover.
\end{prop}

\begin{proof} By Proposition 2.5 and the definition of projective $I$-covers.
\end{proof}

\begin{lem} Let $I$ be an ideal of $R$ and $f: P\rightarrow M$ a
projective $I$-cover. If $Q$ is projective and $g: Q\rightarrow M$
is epic. Then there are decompositions $P=A\oplus B$ and $Q=X\oplus
Y$ such that
\begin{enumerate}
\item $A\cong X$;
\item $f\mid_{A}: A\rightarrow M$ is a projective $I$-cover;
\item $g\mid_{X}: X\rightarrow M$ is a projective $I$-cover;
\item $B\subseteq Kerf$, $Y\subseteq Kerg$.
\end{enumerate}
\end{lem}
\begin{proof} Since $Q$ is projective, there is a homomorphism $h: Q\rightarrow P$
such that $fh=g$, and so $P=h(Q)+Kerf$. Since $Kerf$ is PSD in $P$,
there is a direct summand $B$ of $P$ such that $P=A\oplus B$ with
$A=h(Q), B\subseteq Kerf$. We shall show that $f\mid_{A}:
A\rightarrow M$ is a projective $I$-cover. It is clear that
$Kerf\mid_{A}=A\cap Kerf\subseteq A\cap IP=IA$. Let
$Kerf\mid_{A}+L=A, L\leq A$. Then $P=Kerf\mid_{A}+L\oplus
B=Kerf+L\oplus B$. Since $Kerf$ is PSD in $P$, there is a direct
summand $K$ of $P$ ($K\subseteq Kerf$) such that $P=L\oplus K\oplus
B$, and hence $A=A\cap P=A\cap (L\oplus K\oplus B)=L\oplus (A\cap
(K\oplus B))$. It is easy to see $A\cap (K\oplus B)\subseteq
Kerf\mid_{A}$, so $Kerf\mid_{A}$ is PSD in $A$. Thus $f\mid_{A}:
A\rightarrow M$ is a projective $I$-cover. Since $A$ is projective,
$h: Q\rightarrow A$ splits, and hence there is a homomorphism $q:
A\rightarrow Q$ such that $hq=1_{A}$. So $Q=X\oplus Y, X=q(A),
Y=Kerh$. Since $X=q(A)$, $A\cong X$. Next we show that $g\mid_{X}:
X\rightarrow M$ is a projective $I$-cover. Since
$g(X)=fh(X)=fh(X+Y)=fh(Q)=M$, $g\mid_{X}: X\rightarrow M$ is epic.
We have $Kerg\mid_{X}=q(Kerf\mid_{A})\subseteq q(IA)=IX$. Now we
only show that $Kerg\mid_{X}$ is PSD in $X$. Assume that
$X=Kerg\mid_{X}+N, X\leq N$, then
$A=h(Kerg\mid_{X})+h(N)=Kerf\mid_{A}+h(N)$. Since $Kerf\mid_{A}$ is
PSD in $A$, there is a direct summand $Z$ of $A$ ($Z\subseteq
Kerf\mid_{A})$ such that $A=Z\oplus h(N)$. We know that $h\mid_{X}:
X\rightarrow A$ is isomorphic, so $X=h\mid_{X}^{-1}(Z)\oplus N$. It
is easy to verify that $h\mid_{X}^{-1}(Z)\subseteq Kerg\mid_{X}$, as
required.
\end{proof}
\begin{lem} Let $I$ be an ideal of a ring $R$, $M$ is a projective
$R$-module and $N\leq M$. Consider the following conditions:
\begin{enumerate}
\item $M/N$ has a projective $I$-cover.
\item $M=Y\oplus X, Y\leq N, X\cap N\leq IM$.
\end{enumerate}
Then $(1)\Rightarrow (2), (2)\Rightarrow (1)$ if $M$ is PSD for $I$.
\end{lem}

\begin{proof} $``(1)\Rightarrow (2)"$ Let $f: P\rightarrow M/N$ be a
projective $I$-cover and $g: M\rightarrow M/N$ be the canonical
epimorphism. By Lemma 2.8, $M=Y\oplus X, Y\subseteq Kerg=N$ and
$Kerg\mid_{X}: X\rightarrow M/N$ is a projective $I$-cover, and so
$Kerg\mid_{X}=X\cap N\subseteq IX\subseteq IM$.

$``(2)\Rightarrow (1)"$ Let $f: X\rightarrow M/N$ with $f(x)=x+N$.
It is easy to see that $f: X\rightarrow M/N$ is a projective
$I$-cover by Lemma 2.2 and assumptions.
\end{proof}

With Lemma 2.9, we can give the following characterization of
$I$-semiregular rings related to projective $I$-covers.
\begin{thm} Let $I$ be an ideal of $R$. Consider the following
conditions:
\begin{enumerate}
\item Every finitely presented $R$-module has a projective
$I$-cover.
\item For every finitely generated left ideal $K$ of $R$, $R/K$ has
a projective $I$-cover.
\item Every cyclically presented left $R$-module has a projective
$I$-cover.
\item $R$ is $I$-semiregular.
\end{enumerate}
Then $(1)\Rightarrow (2)\Rightarrow (3)\Rightarrow (4)$,
$(4)\Rightarrow (1)$ if $R$ is a left PSD ring for $I$.

\end{thm}
\begin{proof} By Lemma 2.9, the rest is similar to [1, Theorem 3.11].
\end{proof}

If $I=\delta{(_{R}R)}$ or $J{(R)}$, then $R$ is a left PSD ring for
$I$, and hence Theorem 2.10  gives the characterizations of
$\delta$-semiregular rings [20] and semiregular rings [10]. Since if
$R$ is $Z(_{R}R)$-semiregular, then $Z(_{R}R)=J(R)\subseteq
\delta(_{R}R)$, we have the following result.

\begin{cor} The following statements are equivalent for a ring $R$.
\begin{enumerate}
\item $R$ is $Z(_{R}R)$-semiregular.
\item Every cyclically presented left $R$-module has a projective
$Z(_{R}R)$-cover.
\item Every finitely presented $R$-module has a projective
$Z(_{R}R)$-cover.
\item For every finitely generated left ideal $K$ of $R$, $R/K$ has
a projective $Z(_{R}R)$-cover.

\end{enumerate}

\end{cor}

Since if $I\leq Soc(_{R}R)$, then $R$ is a left PSD ring for $I$,
and hence we have

\begin{cor} The following statements are equivalent for a ring $R$.
\begin{enumerate}
\item $R$ is $Soc(_{R}R)$-semiregular.
\item Every cyclically presented left $R$-module has a projective
$Soc(_{R}R)$-cover.
\item Every finitely presented $R$-module has a projective
$Soc(_{R}R)$-cover.
\item For every finitely generated left ideal $K$ of $R$, $R/K$ has
a projective $Soc(_{R}R)$-cover.

\end{enumerate}

\end{cor}

 Next we shall consider the characterizations of $I$-semiperfect
 rings.
\begin{thm} Let $I$ be an ideal of $R$. Consider the following
conditions:
\begin{enumerate}
\item Every finitely generated $R$-module has a projective
$I$-cover.
\item Every factor module of $_{R}R$ has
a projective $I$-cover.
\item For every countably generated left ideal $K$ of $R$, $R/K$ has
a projective $I$-cover.
\item $R$ is $I$-semiperfect.
\item Every simple $R$-module has a projective
$I$-cover.
\item Every simple factor module of $_{R}R$ has a projective
$I$-cover.
\end{enumerate}
Then $(1)\Rightarrow (2)\Rightarrow (3)\Rightarrow (4)$ and
$(1)\Rightarrow (5)\Rightarrow (6)$, $(4)\Rightarrow (1)$ if $R$ is
a left PSD ring for $I$; and $(6)\Rightarrow (4)$ if $I$ is PSD in
$_{R}R$.

\end{thm}
\begin{proof} Similar to [1, Theorem 4.8] by Lemma 2.9.
\end{proof}

When $I=\delta(_{R}R)$ or $Soc(_{R}R)$ or $J(R)$, Theorem 2.13 gives
the characterizations of ($\delta$-, $Soc(_{R}R)$-) semiperfect
rings (See [20], [11], [10]).

\begin{cor} The following statements are equivalent for a ring $R$.
\begin{enumerate}
\item $R$ is $Z(_{R}R)$-semiperfect.
\item Every finitely generated $R$-module has a projective
$Z(_{R}R)$-cover.
\item Every factor module of $_{R}R$ has
a projective $Z(_{R}R)$-cover.
\item For every countably generated left ideal $K$ of $R$, $R/K$ has
a projective $Z(_{R}R)$-cover.
\item Every simple $R$-module has a projective
$Z(_{R}R)$-cover.
\item Every simple factor module of $_{R}R$ has a projective
$Z(_{R}R)$-cover.

\end{enumerate}

\end{cor}

\section{\bf $I$-semiregular($I$-semiperfect) rings and locally projective $I$-covers}

Ding and Chen [5] and Xue [18] gave some characterizations of rings
by locally projective covers. Inspired by those, we introduce the
notion of locally projective $I$-covers and use it to characterize
$I$-semiregular and $I$-semiperfect rings in this section. Firstly,
we recall some definitions and facts. A module $P$ is called locally
projective [18, 21] in case it satisfies any of the following
equivalent condition: $(a)$ if $A$ and $B$ are modules, $g:
A\rightarrow B$ is an epimorphism and $f: P\rightarrow B$ is a
homomorphism then for every finitely generated (cyclic) submodule
$P_{0}$ of $P$ there is a homomorphism $h: P\rightarrow A$ such that
$f\mid_{P_{0}}=gh\mid_{P_{0}}$; $(b)$ if $M$ is a module and $f:
M\rightarrow P$ is an epimorphism then for every finitely generated
(cyclic) submodule $P_{0}$ of $P$ there is a homomorphism $h:
P\rightarrow M$ such that $fg\mid_{P_{0}}=1\mid_{P_{0}}$. Clearly,
every finitely generated locally projective module is projective.
The following facts are well known. (1) a direct sum of locally
projective modules is locally projective; (2) any direct summand of
a locally projective module is locally projectvie; (3) if $P$ is a
locally projective module, then (1) $Rad(P)=J(R)P$; (2) if
$Rad(P)=P$, then $P=0$. We also recall that a pair $(P, f)$ is
called a locally projective ($\delta$-) cover of $M$ if $P$ is
locally projective, $f$ is an epimorphism from $P$ to $M$ such that
$Kerf\ll P$ ($Kerf\ll_{\delta} P$).

\begin{defn} A pair $(P, f)$ is called a locally projective
$I$-cover of $M$ if $P$ is locally projective, $f$ is an epimorphism
from $P$ to $M$ such that $Kerf\leq IP$, and $Kerf$ is PSD in $P$.
\end{defn}

It is easy to see that a module $M$ has a locally projective
$0$-cover if and only if $M$ is locally projective.

\begin{prop} If each $f_{i}: P_{i}\rightarrow M_{i}, (i=1, 2,
\cdot\cdot\cdot, n)$ is a locally projective $I$-cover, then $\oplus
_{i=1}^{n}f_{i}: \oplus _{i=1}^{n}P_{i}\rightarrow \oplus
_{i=1}^{n}M_{i}$ is a locally projective $I$-cover.

\end{prop}

\begin{prop} A module $M$ has a locally projective $J(R)$-cover if
and only if $M$ has a locally projective cover.
\end{prop}

\begin{proof} $``\Leftarrow"$ is clear.

$``\Rightarrow"$ Let $f: P\rightarrow M$ be a locally projective
$J(R)$-cover. Then $Kerf\subseteq J(R)P$, $Kerf$ is PSD in $P$. Next
we shall show that $Kerf\ll P$. Let $Kerf+L=P, L\leq M$. Since
$Kerf$ is PSD in $P$, there is a summand $K$ of $P$ with $K\leq
Kerf$ such that $K\oplus L=P$. Since $Rad(K)\oplus Rad(L)=Rad(P)$
and $K\subseteq Rad(P)$, $Rad(K)=K$. Since $K$ is locally
projective, $K=0$, and hence $L=P$, as desired.
\end{proof}

The following lemma is a key result of this section.

\begin{lem} Let $f: P\rightarrow M$ be a locally projective
$I$-cover. If $M$ is finitely generated, then $f: P\rightarrow M$ is
a projective $I$-cover.
\end{lem}

\begin{proof} It suffices to prove that $P$ is projective.
Since $f: P\rightarrow M$ is a locally projective $I$-cover and $M$
is finitely generated, there is a finitely generated submodule
$P_{0}$ of $P$ such that $P_{0}+Kerf=P$. Note that $Kerf$ is PSD in
$P$, there is a projective summand $K$ of $P$ with $K\subseteq Kerf$
such that $P_{0}\oplus K=P$. Since $P$ is locally projective,
$P_{0}$ is locally projective. Since $P_{0}$ is finitely generated,
$P_{0}$ is projective. Thus $P$ is projective, as required.

\end{proof}

\begin{thm} Let $I$ be an ideal of $R$. Consider the following
conditions:
\begin{enumerate}
\item Every finitely presented $R$-module has a locally projective
$I$-cover.
\item For every finitely generated left ideal $K$ of $R$, $R/K$ has
a locally projective $I$-cover.
\item Every cyclically presented left $R$-module has a locally projective
$I$-cover.
\item $R$ is $I$-semiregular.
\end{enumerate}
Then $(1)\Rightarrow (2)\Rightarrow (3)\Rightarrow (4)$,
$(4)\Rightarrow (1)$ if $R$ is a left PSD ring for $I$.

\end{thm}
\begin{proof} $``(1)\Rightarrow (2)\Rightarrow (3)"$ are clear.

$``(3)\Rightarrow (4)"$ It follows by Lemma 3.4 and  Theorem 2.10.

$``(4)\Rightarrow (1)"$ is clear by Theorem 2.10.
\end{proof}

If $I=\delta{(_{R}R)}$ or $J{(R)}$, then $R$ is a left PSD ring for
$I$.

\begin{cor}  The following conditions are equivalent for a ring $R$.
\begin{enumerate}

\item $R$ is semiregular.

\item Every finitely presented $R$-module has a locally projective
cover.
\item For every finitely generated left ideal $K$ of $R$, $R/K$ has
a locally projective cover.
\item Every cyclically presented left $R$-module has a locally projective
cover.

\end{enumerate}

\end{cor}

\begin{cor}  The following conditions are equivalent for a ring $R$.
\begin{enumerate}

\item $R$ is $\delta$-semiregular.

\item Every finitely presented $R$-module has a locally projective
$\delta$-cover.
\item For every finitely generated left ideal $K$ of $R$, $R/K$ has
a locally projective $\delta$-cover.
\item Every cyclically presented left $R$-module has a locally projective
$\delta$-cover.

\end{enumerate}

\end{cor}

\begin{proof} It follows by Theorem 2.10 and Lemma 3.4.
\end{proof}

 Since if $R$ is $Z(_{R}R)$-semiregular, then
$Z(_{R}R)=J(R)\subseteq \delta(_{R}R)$, we have the following
result.

\begin{cor} The following statements are equivalent for a ring $R$.
\begin{enumerate}
\item $R$ is $Z(_{R}R)$-semiregular.
\item Every cyclically presented left $R$-module has a locally projective
$Z(_{R}R)$-cover.
\item Every finitely presented $R$-module has a locally projective
$Z(_{R}R)$-cover.
\item For every finitely generated left ideal $K$ of $R$, $R/K$ has
a locally projective $Z(_{R}R)$-cover.

\end{enumerate}

\end{cor}

Since if $I\leq Soc(_{R}R)$, then $R$ is a left PSD ring for $I$,
and hence we have

\begin{cor} The following statements are equivalent for a ring $R$.
\begin{enumerate}
\item $R$ is $Soc(_{R}R)$-semiregular.
\item Every cyclically presented left $R$-module has a locally projective
$Soc(_{R}R)$-cover.
\item Every finitely presented $R$-module has a locally projective
$Soc(_{R}R)$-cover.
\item For every finitely generated left ideal $K$ of $R$, $R/K$ has
a locally projective $Soc(_{R}R)$-cover.

\end{enumerate}
\end{cor}
 Next we shall consider the characterizations of $I$-semiperfect
 rings.
\begin{thm} Let $I$ be an ideal of $R$. Consider the following
conditions:
\begin{enumerate}
\item Every finitely generated $R$-module has a locally projective
$I$-cover.
\item Every factor module of $_{R}R$ has
a locally projective $I$-cover.
\item For every countably generated left ideal $K$ of $R$, $R/K$ has
a locally projective $I$-cover.
\item $R$ is $I$-semiperfect.
\item Every simple $R$-module has a locally projective
$I$-cover.
\item Every simple factor module of $_{R}R$ has a locally projective
$I$-cover.
\end{enumerate}
Then $(1)\Rightarrow (2)\Rightarrow (3)\Rightarrow (4)$ and
$(1)\Rightarrow (5)\Rightarrow (6)$, $(4)\Rightarrow (1)$ if $R$ is
a left PSD ring for $I$; and $(6)\Rightarrow (4)$ if $I$ is PSD in
$_{R}R$.

\end{thm}
\begin{proof} It follows by Theorem 2.13 and Lemma 3.4.
\end{proof}

When $I=\delta(_{R}R)$ or $J(R)$, we have the following.

\begin{cor} The following statements are equivalent for a ring $R$.
\begin{enumerate}
\item $R$ is semiperfect.

\item Every finitely generated $R$-module has a locally projective
cover.
\item Every factor module of $_{R}R$ has
a locally projective cover.
\item For every countably generated left ideal $K$ of $R$, $R/K$ has
a locally projective cover.

\item Every simple $R$-module has a locally projective
cover.
\item Every simple factor module of $_{R}R$ has a locally projective
cover.
\end{enumerate}

\end{cor}

\begin{cor} The following statements are equivalent for a ring $R$.
\begin{enumerate}
\item $R$ is $\delta$-semiperfect.

\item Every finitely generated $R$-module has a locally projective
$\delta$-cover.
\item Every factor module of $_{R}R$ has
a locally projective $\delta$-cover.
\item For every countably generated left ideal $K$ of $R$, $R/K$ has
a locally projective $\delta$-cover.

\item Every simple $R$-module has a locally projective
$\delta$-cover.
\item Every simple factor module of $_{R}R$ has a locally projective
$\delta$-cover.
\end{enumerate}

\end{cor}

\begin{proof} By Lemma 3.4 and Theorem 3.10.
\end{proof}

\begin{cor} The following statements are equivalent for a ring $R$.
\begin{enumerate}
\item $R$ is $Z(_{R}R)$-semiperfect.
\item Every finitely generated $R$-module has a locally projective
$Z(_{R}R)$-cover.
\item Every factor module of $_{R}R$ has
a locally projective $Z(_{R}R)$-cover.
\item For every countably generated left ideal $K$ of $R$, $R/K$ has
a locally projective $Z(_{R}R)$-cover.
\item Every simple $R$-module has a locally projective
$Z(_{R}R)$-cover.
\item Every simple factor module of $_{R}R$ has a locally projective
$Z(_{R}R)$-cover.

\end{enumerate}

\end{cor}

\begin{cor} The following statements are equivalent for a ring $R$.
\begin{enumerate}
\item $R$ is $Soc(_{R}R)$-semiperfect.
\item Every finitely generated $R$-module has a locally projective
$Soc(_{R}R)$-cover.
\item Every factor module of $_{R}R$ has
a locally projective $Soc(_{R}R)$-cover.
\item For every countably generated left ideal $K$ of $R$, $R/K$ has
a locally projective $Soc(_{R}R)$-cover.
\item Every simple $R$-module has a locally projective
$Soc(_{R}R)$-cover.
\item Every simple factor module of $_{R}R$ has a locally projective
$Soc(_{R}R)$-cover.

\end{enumerate}

\end{cor}

\section{\bf $I$-semiregular($I$-semiperfect) rings characterized by projectivity
classes}

Wang [14] gave characterizations of semiregular rings and
semiperfect rings by introducing the concept of projectivity classes
of modules. Motivated by this, we shall characterize $I$-semiregular
rings and $I$-semiperfect rings in term of projectivity classes of
modules in this section.

\begin{defn} ( [14] ) A class $\mathscr{P}$ of $R$-modules is called a projectivity class
 if it contains all self-projective modules and for every module
$M$ and every projective module $P$ with an epimorphism $f:
P\rightarrow M$, $P\oplus M\in \mathscr{P}$ implies that $M$ is
projective.
\end{defn}

\begin{exa}
1. The class of all quasi-projective modules is a projectivity
class.

2. The class of all weakly quasi-projective modules in the sense of
Rangaswamy  and Vanaja is a projectivity class.

3. The class of all pseudo-projective modules is a projectivity
class.

4. The class of all direct-projective modules is a projectivity
class.

5. For any perfect ring $R$, the class of all discrete $R$-modules
is a projectivity class.

\end{exa}

For a projectivity class $\mathscr{P}$, we introduce the following
concept.

\begin{defn} Let $I$ be an ideal of a ring $R$ and $M$ be a module.
 We call an epimorphism $f: P\rightarrow M$ a
$\mathscr{P}$-projective $I$-cover of $M$ if $P\in \mathscr{P}$ and
$Kerf\leq IP$, $Kerf$ is PSD  in $P$.
\end{defn}

\begin{lem} Let $I$ be an ideal of a ring $R$, $\mathscr{P}$ a projectivity class and which is closed
under taking direct summands. Suppose that $P$ is a projective
module and there is an epimorphism $f: P\rightarrow M$. If $P\oplus
M$ has a $\mathscr{P}$-projective $I$-cover, then $M$ has a
projective $I$-cover.
\end{lem}
\begin{proof} Let $g: Q\rightarrow P\oplus M$ be a $\mathscr{P}$-projective
$I$-cover of $P\oplus M$. We have an exact sequence $0\rightarrow
g^{-1}(M)\rightarrow Q\stackrel{\phi g}\rightarrow P\rightarrow 0$,
where $\phi: P\oplus M\rightarrow P$ is the projection map. Since
$P$ is projective, $Q\cong P\oplus g^{-1}(M)$ and Ker$(\phi
g)=g^{-1}(M)$ is a direct summand of $Q$. Note that Kerg is PSD in
$Q$, and so Kerg is PSD in $g^{-1}(M)$ by Lemma 2.2. It is easy to
see that Ker$g\subseteq Ig^{-1}(M)$. Clearly, we have an exact
sequence $0\rightarrow $Ker$g\rightarrow
g^{-1}(M)\stackrel{g}\rightarrow M\rightarrow 0$. So it suffices to
show that $g^{-1}(M)$ is projective. Since $P$ is projective with an
epimorphism $f: P\rightarrow M$, and $g: g^{-1}(M)\rightarrow M$ is
an epimorphism, there is a homomorphism $\alpha: P\rightarrow
g^{-1}(M)$ such that $g\alpha=f$, and hence
Im$\alpha+$Ker$g=g^{-1}(M)$. Since Ker$g$ is PSD in $g^{-1}(M)$,
there is a projective submodule $L$ of Ker$g$ such that
Im$\alpha\oplus L=g^{-1}(M)$. Thus $Q\cong P\oplus$Im$\alpha\oplus
L$ belongs to $\mathscr{P}$. Since $\mathscr{P}$ is closed under
taking direct summands, $P\oplus$Im$\alpha\in \mathscr{P}$, and
there is an epimorphism $P\rightarrow$Im$\alpha$, Im$\alpha$ is
projective. Thus $g^{-1}(M)=$Im$\alpha\oplus L$ is projective, as
desired.
\end{proof}
\begin{cor} Let $I$ be an ideal of a ring $R$, $\mathscr{P}$ a projectivity class and which is closed
under taking direct summands.  Consider the following statements:
\begin{enumerate}
\item Every finitely presented $R$-module has a $\mathscr{P}$-projective
$I$-cover.
\item For every finitely generated left ideal $K$ of $R$, $R/K$ has
a $\mathscr{P}$-projective $I$-cover.
\item Every cyclically presented left $R$-module has a $\mathscr{P}$-projective
$I$-cover.
\item $R$ is $I$-semiregular.
\end{enumerate}
Then $(1)\Rightarrow (2)\Rightarrow (3)\Rightarrow (4)$,
$(4)\Rightarrow (1)$ if $R$ is a left PSD ring for $I$.

\end{cor}
\begin{proof} $``(1)\Rightarrow (2)\Rightarrow (3)"$ are clear.

$``(3)\Rightarrow (4)"$ $R\oplus R/Rr$ has a
$\mathscr{P}$-projective by assumption, and hence $R/Rr$ has a
projective $I$-cover by Lemma 4.4. Thus $R$ is $I$-semiregular by
Theorem 2.10. $``(4)\Rightarrow (1)"$ is clear by Theorem 2.10.
\end{proof}
\begin{cor} Let $R$ be a ring, $\mathscr{P}$ a projectivity class and which is closed
under taking direct summands. Then the following statements are
equivalent.
\begin{enumerate}
\item $R$ is $Z(_{R}R)$-semiregular.
\item Every cyclically presented left $R$-module has a $\mathscr{P}$-projective
$Z(_{R}R)$-cover.
\item Every finitely presented $R$-module has a $\mathscr{P}$-projective
$Z(_{R}R)$-cover.
\item For every finitely generated left ideal $K$ of $R$, $R/K$ has
a $\mathscr{P}$-projective $Z(_{R}R)$-cover..
\end{enumerate}
\end{cor}
\begin{cor}
Let $R$ be a ring, $\mathscr{P}$ a projectivity class and which is
closed under taking direct summands. Then the following statements
are equivalent.
\begin{enumerate}
\item $R$ is $Soc(_{R}R)$-semiregular.
\item Every cyclically presented left $R$-module has a $\mathscr{P}$-projective
$Soc(_{R}R)$-cover.
\item Every finitely presented $R$-module has a $\mathscr{P}$-projective
$Soc(_{R}R)$-cover.
\item For every finitely generated left ideal $K$ of $R$, $R/K$ has
a $\mathscr{P}$-projective $Soc(_{R}R)$-cover.
\end{enumerate}
\end{cor}

\begin{cor} Let $\mathscr{P}$ be a projectivity class and which is closed
under taking direct summands. The following statements are
equivalent for a ring $R$:
\begin{enumerate}
\item $R$ is $\delta$-semiregular.
\item Every cyclically presented left $R$-module has a $\mathscr{P}$-projective
$\delta$-cover.
\item Every finitely presented $R$-module has a $\mathscr{P}$-projective
$\delta$-cover.
\item For every finitely generated left ideal $K$ of $R$, $R/K$ has
a $\mathscr{P}$-projective $\delta$-cover.
\end{enumerate}
\end{cor}

When $\mathscr{P}$ is the class of all direct-projective modules,
Corollary 4.8 gives [15, Proposition 4.4].

\begin{cor} Let $\mathscr{P}$ be a projectivity class and which is closed
under taking direct summands. The following statements are
equivalent for a ring $R$:
\begin{enumerate}
\item $R$ is semiregular.
\item Every cyclically presented left $R$-module has a $\mathscr{P}$-projective
cover.
\item Every finitely presented $R$-module has a $\mathscr{P}$-projective
cover.
\item For every finitely generated left ideal $K$ of $R$, $R/K$ has
a $\mathscr{P}$-projective cover.
\end{enumerate}
\end{cor}

When $\mathscr{P}$ is the class of all direct-projective modules,
Corollary 4.9 gives [17, Corollary 3.4].

\begin{thm} Let $I$ be an ideal of a ring $R$, $\mathscr{P}$ a projectivity class and which is closed
under taking direct summands.  Consider the following statements:
\begin{enumerate}
\item Every finitely generated $R$-module has a $\mathscr{P}$-projective
$I$-cover.
\item Every factor module of $_{R}R$ has
a $\mathscr{P}$-projective $I$-cover.
\item For every countably generated left ideal $K$ of $R$, $R/K$ has
a $\mathscr{P}$-projective $I$-cover.
\item $R$ is $I$-semiperfect.

\item Every simple $R$-module has a $\mathscr{P}$-projective
$I$-cover.
\item Every simple factor module of $_{R}R$ has a $\mathscr{P}$-projective
$I$-cover.
\end{enumerate}
Then $(1)\Rightarrow (2)\Rightarrow (3)\Rightarrow (4)$ and
$(1)\Rightarrow (5)\Rightarrow (6)$, $(4)\Rightarrow (1)$ if $R$ is
a left PSD ring for $I$; and $(6)\Rightarrow (4)$ if $I$ is PSD in
$_{R}R$.
\end{thm}
\begin{proof} Following by Theorem 2.13 and Lemma 3.4.
\end{proof}

\begin{cor} Let $I$ be an ideal of a ring $R$, $\mathscr{P}$ a projectivity class and which is closed
under taking direct summands. Then the following statements are
equivalent.
\begin{enumerate}
\item $R$ is $Z(_{R}R)$-semiperfect.
\item Every finitely generated $R$-module has a $\mathscr{P}$-projective
$Z(_{R}R)$-cover.
\item Every factor module of $_{R}R$ has
a $\mathscr{P}$-projective $Z(_{R}R)$-cover.
\item For every countably generated left ideal $K$ of $R$, $R/K$ has
a $\mathscr{P}$-projective $Z(_{R}R)$-cover.
\item Every simple $R$-module has a $\mathscr{P}$-projective
$Z(_{R}R)$-cover.
\item Every simple factor module of $_{R}R$ has a $\mathscr{P}$-projective
$Z(_{R}R)$-cover.
\end{enumerate}

\end{cor}

\begin{cor} Let $I$ be an ideal of a ring $R$, $\mathscr{P}$ a projectivity class and which is closed
under taking direct summands. Then the following statements are
equivalent.
\begin{enumerate}
\item $R$ is $Soc(_{R}R)$-semiperfect.
\item Every finitely generated $R$-module has a $\mathscr{P}$-projective
$Soc(_{R}R)$-cover.
\item Every factor module of $_{R}R$ has
a $\mathscr{P}$-projective $Soc(_{R}R)$-cover.
\item For every countably generated left ideal $K$ of $R$, $R/K$ has
a $\mathscr{P}$-projective $Soc(_{R}R)$-cover.
\item Every simple $R$-module has a $\mathscr{P}$-projective
$Soc(_{R}R)$-cover.
\item Every simple factor module of $_{R}R$ has a $\mathscr{P}$-projective
$Soc(_{R}R)$-cover.

\end{enumerate}
\end{cor}

\begin{cor} Let $\mathscr{P}$ be any projectivity class and which is
closed under taking direct summands. The following statements are
equivalent for a ring $R$.
\begin{enumerate}
\item $R$ is a left $\delta$-semiperfect ring.

\item Every finitely generated $R$-module has a $\mathscr{P}$-projective
$\delta$-cover.
\item Every factor module of $_{R}R$ has
a $\mathscr{P}$-projective $\delta$-cover.
\item For every countably generated left ideal $K$ of $R$, $R/K$ has
a $\mathscr{P}$-projective $\delta$-cover.

\item Every simple $R$-module has a $\mathscr{P}$-projective
$\delta$-cover.
\item Every simple factor module of $_{R}R$ has a $\mathscr{P}$-projective
$\delta$-cover.

\end{enumerate}

\end{cor}

\begin{cor} Let $\mathscr{P}$ be any projectivity class and which is
closed under taking direct summands. The following statements are
equivalent for a ring $R$.
\begin{enumerate}
\item $R$ is a left semiperfect ring.

\item Every finitely generated $R$-module has a $\mathscr{P}$-projective
cover.
\item Every factor module of $_{R}R$ has
a $\mathscr{P}$-projective cover.
\item For every countably generated left ideal $K$ of $R$, $R/K$ has
a $\mathscr{P}$-projective cover.
\item Every simple $R$-module has a $\mathscr{P}$-projective
cover.
\item Every simple factor module of $_{R}R$ has a $\mathscr{P}$-projective
cover.

\end{enumerate}

\end{cor}

\section{\bf $I$-semiregular($I$-semiperfect) rings and $I$-supplemented modules}

It is well know that a ring $R$ is semiperfect if and only if
$R_{R}$ is a supplemented module if and only if $_{R}R$ is a
supplemented module. We also know that a ring $R$ is semiregular if
and only if $R_{R}$ is a finitely supplemented module if and only if
$_{R}R$ is a finitely supplemented module. Here we introduce the
notion of $I$-supplemented modules and use it to characterize
$I$-semiregular(semiperfect) rings.

Let $R$ be a ring $I$ an ideal of $R$, $M$ a module and $N, L\leq
M$. $N$ is called a \emph{supplement} of $L$ in $M$ if $N + L = M$
and $N$ is minimal with respect to this property. Equivalently, $M =
N + L$ and $N\cap L\ll N$. A module $M$ is called
\emph{supplemented} if every submodule of $M$ has a supplement in
$M$. $N$ is called a \emph{$\delta$-supplement} [6] of $L$ if
$M=N+L$ and $N\cap L\ll_{\delta}N$. $M$ is called a
\emph{$\delta$-supplemented module} if every submodule of $M$ has a
$\delta$-supplement. A module $M$ is said to be
\emph{$\delta$-lifting} [6] if for any submodule $N$ of $M$, there
exists a direct summand $K$ of $M$ such that $K\leq N$ and
$N/K\ll_{\delta} M/K$, equivalently, for every submodule $N$ of $M$,
$M$ has a decomposition with $M=M_1\oplus M_2$, $M_1\leq N$ and
$M_2\cap N$ is $\delta$-small in $M_2$. $N$ is DM in $M$ [1] if
there is a summand $S$ of $M$ such that $S\leq N$ and $M=S+X$,
whenever $N+X=M$ for a submodule $X$ of $M$. $M$ is DM for $I$ if
any submodule of $IM$ is DM in $M$. $R$ is a left DM ring for $I$ if
for any finitely generated free left $R$-module is DM for $I$.

\begin{defn} Let $R$ be a ring and $I$ an ideal of $R$, $M$ a
module. $M$ is called an $I$-supplemented module (finitely
$I$-supplemented module) if for every submodule (finitely generated
submodule ) $X$ of $M$, there is a projective submodule $Y$ of $M$
such that $X+Y=M$, $X\cap Y\subseteq IY$ and $X\cap Y$ is DM in $Y$.
\end{defn}

\begin{thm}  Let $R$ be a ring.
The following statements are equivalent for a projective
module $M$.

\begin{enumerate}
\item $M$ is a $J(R)$-supplemented module (a $\delta(_{R}R)$-supplemented module,
respectively).
\item $M$ is a supplemented module ( a $\delta$-supplemented module,
respectively).

\end{enumerate}
\end{thm}
\begin{proof} $``(1)\Rightarrow (2)"$ Let $M$ be a $J(R)$-supplemented
module (a $\delta(_{R}R)$-supplemented module, respectively). Then
for every submodule $X$ of $M$, there is a projective submodule $Y$
of $M$ such that $X+Y=M$ , $X\cap Y\subseteq J(R)Y$ ($X\cap
Y\subseteq \delta(_{R}R)Y$) and $X\cap Y$ is DM in $Y$. Next we
shall show that $X\cap Y\ll Y$ ($X\cap Y\ll_{\delta} Y$). Assume
that $X\cap Y+L=Y, L\leq Y$. Note that $X\cap Y$ is DM in $Y$, there
is a summand $K$ of $Y$ which is contained in $X\cap Y$ such that
$K+L=Y$. Write $Y=K\oplus K', K'\leq M$,  then $Rad(K)\oplus
Rad(K')=Rad(Y)$ ($\delta(K)\oplus \delta(K')=\delta(Y)$). Since
$K\subseteq Rad(Y)$ ($K\subseteq \delta(Y)$), $K=Rad(K)$
($K=\delta(K)$). Since $K$ is projective, $K=0$ ($K$ is semisimple,
thus $X\cap Y\ll_{\delta} Y$ by [20, Lemma 1.2]), and hence $L=Y$,
so $X\cap Y\ll Y$.

$``(2)\Rightarrow (1)"$ Let $M$ be a supplemented module. Then for
every submodule $X$ of $M$, there a submodule $Y$ of $M$ such that
$X+Y=M$ and $X\cap Y\ll Y$. Since $M$ is projective, $Y$ is a direct
summand of $M$, and hence $Y$ is projective. It is clear that $X\cap
Y\subseteq Rad(Y)=J(R)Y$ and $X\cap Y$ is DM in $Y$. (Let $M$ be a
$\delta$-supplemented module. Since $M$ is projective, $M$ is
$\delta$-lifting. Thus for every submodule $X$ of $M$, there is a
direct summand $Y$ of $M$ such that $M=X+Y$ and $X\cap Y\ll_{\delta}
Y$. The rest is obvious.)

\end{proof}

\begin{thm} Let $R$ be a ring and $I$ an ideal of $R$, $M$ a projective
module. Consider the following conditions:
\begin{enumerate}
\item $M$ is an $I$-supplemented module.
\item $M$ is an $I$-semiperfect module.
\end{enumerate}
Then $(1)\Rightarrow (2)$,  and $(2)\Rightarrow (1)$ if $M$ is DM
for $I$.
\end{thm}
\begin{proof} $``(1)\Rightarrow (2)"$ Let $M$ be an $I$-supplemented
module, then for every submodule $X$ of $M$, there is a projective
submodule $Y$ of $M$ such that $M=X+Y$, $X\cap Y\subseteq IY$ and
$X\cap Y$ is DM in $Y$. We define $f: Y\rightarrow M/X$ be such that
$f(y)=y+X$. Then $f$ is an epimorphism with $Kerf=X\cap Y$, and
hence $Y$ is a projective $I$-cover (in the sense of [1]) of $M/X$.
The rest is obvious by [1, Lemma 3.10].

$``(2)\Rightarrow (1)"$ Let $M$ be an $I$-semiperfect module, then
for every submodule $X$ of $M$, there is a decomposition $M=A\oplus
Y$ such that $A$ is projective, $A\subseteq X$ and $X\cap Y\subseteq
IM$. Thus $M=X+Y$, $Y$ is projective, $X\cap Y\subseteq IY$. Since
$M$ is DM for $I$, $X\cap Y$ is DM in $Y$ by [1, Lemma 3.2], as
desired.

\end{proof}

\begin{cor} Let $M$ be a projective
module with $Rad(M)\ll M$ ($\delta(M)\ll_{\delta}M$). Then
 $M$ is a ($\delta$-)supplemented module if and only if
 $M$ is a ($\delta$-)semiperfect module.
\end{cor}

\begin{cor} Let $R$ be a left DM ring and $I$ an ideal of $R$. Then $R$ is an
$I$-semiperfect ring if and only if $_{R}R$ is an $I$-supplemented
module.
\end{cor}

Write $I=J(R)$ or $\delta(_{R}R)$ in Corollary 5.5, since $R$ is a
left DM ring, we have the following.

\begin{cor} The following statements are equivalent for a ring $R$.
\begin{enumerate}
\item $R$ is ($\delta$-) semiperfect;
\item $_{R}R$ is a ($\delta$-) supplemented module;
\item $R_{R}$ is a ($\delta$-) supplemented module;
\item $_{R}R$ is a $J(R)$($\delta(_{R}R)$)-supplemented module;
\item $R_{R}$ is a $J(R)$($\delta(_{R}R)$)-supplemented module.
\end{enumerate}
\end{cor}

\begin{proof} It follows by Theorem 5.2 and 5.3.
\end{proof}

Similarly, we obtain the following results.

\begin{thm} Let $R$ be a left DM ring and $I$ an ideal of $R$. Then $R$ is an
$I$-semiregular ring if and only if $_{R}R$ is a finitely
$I$-supplemented module if and only if $R_{R}$ is a finitely
$I$-supplemented module.
\end{thm}
\begin{cor} The following statements are equivalent for a ring $R$.
\begin{enumerate}
\item $R$ is ($\delta$-) semiregular.
\item $_{R}R$ is a finite ($\delta$-) supplemented module.
\item $R_{R}$ is a finite ($\delta$-) supplemented module.
\item $_{R}R$ is a finite $J(R)$($\delta(_{R}R)$)-supplemented
module.
\item $R_{R}$ is a finite $J(R)$($\delta(_{R}R)$)-supplemented module.
\end{enumerate}
\end{cor}

Since if $I\subseteq Soc(_{R}R)$, then $R$ is a left DM ring, we
have
\begin{cor} A ring $R$ is $Soc(_{R}R)$-semiperfect if and
only if $_{R}R$ is a  $Soc(_{R}R)$-supplemented module.
\end{cor}
\begin{cor} A ring $R$ is $Soc(_{R}R)$-semiregular if and
only if $_{R}R$ is a  finitely $Soc(_{R}R)$-supplemented module if
and only if $R_{R}$ is a finitely $Soc(R_{R})$-supplemented module.
\end{cor}


\begin{thebibliography}{99}

\bibitem{A} M. Alkan, W. K. Nicholson, A. C. Ozcan, A generalization
of projective covers, J. Algebra 319 (2008) 4947-4960.
\bibitem{A} M. Alkan,  A. C. Ozcan, Semiregular modules with respect
to a fully invariant submodule, Comm. Algebra 32(11) (2004)
4285-4301.




\bibitem{AF}F. W. Anderson and K. R. Fuller, Rings and Categories
of Modules, Springer-Verlag, New York, 1974.
\bibitem{8} J. Clark, C. Lomp,  N. Vanaja and R. Wisbauer, Lifting
modules, IN:Supplements and Projectivity in Module Theory Sereis:
Frontiers in Mathematics, 2006.

\bibitem{DC}  N. Q. Ding and J. L. Chen, On a Characterization of Perfect Rings,
Comm. Algebra, 27 (1999) 785-791.

\bibitem{K} M. T. Kosan, $\delta$-lifting and
$\delta$-supplemented modules, Algebra Colloq. 14 (2007) 53-60.

\bibitem{M} S. H. Mohamed and B. J. M\"{u}ller,
Continuous and Discrete Modules, London Math. Soc.; LNS147,
Cambridge Univ. Press: Cambridge, 1990.




\bibitem{N} W. K. Nicholson, Semiregular modules and rings, Can. J. Math.
5 (1976) 1105-1120.

\bibitem{N} W. K. Nicholson, Strong lifting, J. Algebra 285 (2005)
795-818.

\bibitem{N} W. K. Nicholson, M. F. Yousif, Quasi-Frobenius
Rings, Cambridge Tracts in Math., vol.158, Cambridge University
Press, Cambridge, 2003.

\bibitem{A} A. C. Ozcan, M. Alkan, Semiperfect modules with respect
to a preradical, Comm. Algebra 34 (2006) 841-856.

\bibitem{R} K. M. Rangaswamy and N. Vanaja, Quasi projectives in Abelian and
module categories,  Pacific J. math 43(1) (1972) 221-238.

\bibitem{T} A. Tuganbaev, Rings Close to Regular, Kluwer Academic
Publishers, Dordrecht, 2002. ¡¡


\bibitem{W2} D. G. Wang, Rings characterized by pojectivity classes,  Comm.
Algebra 25(1) (1997) 105-116.


\bibitem{W4} Y. D. Wang, Relatively semiregular modules, Algebra
Colloq. 13 (2006) 647-654.


\bibitem{W5} R. Wisbauer, Foundations of Module and Ring
Theory, Gordon and Breach, Philadelphia, 1991.

\bibitem{X} W. M. Xue, Semiregular modules and F-semiperfect modules, Comm. Algebra 23(3) (1995) 1035-1046.
\bibitem{X2} W. M. Xue, Characterizations of Semiperfect and Perfect
Rings, Publ. Mat. 40  (1996) 115-125.

\bibitem{Z} M. F. Yousif, Y. Q. Zhou, Semiregular,semiperfect and
perfect rings relative to an ideal, Rocky Mountain J. Math. 32
(2002) 1651-1671.


\bibitem{Z} Y. Q. Zhou, Generalizations of perfect, semiperfect,
and semiregular rings, Algebra Colloq. 3 (2000) 305-318.
\bibitem{Z} B. Zimmermann-Huisgen, Pure submodules of direct
products of free modules, Math.  Ann. 224 (1976) 233-245.

\bibitem{Z} B. Zimmermann-Huisgen, Direct products of modules and
algebraic compactness, Habilitationsschrift, Tech. Univ. Munich,
1980.
\end{thebibliography}
\end{document}